\documentclass[12pt,reqno]{amsart}
\usepackage{fullpage}
\usepackage[top=1.5in, bottom=2in, left=2.5in, right=2.5in]{geometry}
\usepackage{graphicx,amsmath,amssymb,amsfonts,amsthm,enumitem,bm,xcolor}
\usepackage{fullpage}
\usepackage{cleveref}

\usepackage{amsmath}

\usepackage[normalem]{ulem}

\usepackage{url}

\usepackage{color}

\renewcommand{\a}{\alpha}
\renewcommand{\b}{\beta}

\newcommand{\G}{\Gamma}
\renewcommand{\l}{\lambda}

\renewcommand{\o}{{\omega}}

\renewcommand{\r}{{\rho}}
\renewcommand{\t}{\tau}

\newcommand{\z}{\zeta}

\renewcommand{\(}{\left\(}
\renewcommand{\)}{\right\)}
\renewcommand{\i}{\infty}

\numberwithin{equation}{section}
\theoremstyle{plain}
\newtheorem{theorem}{Theorem}[section]

\newtheorem{conjecture}[theorem]{Conjecture}
\newtheorem{corollary}[theorem]{Corollary}
\newtheorem{proposition}[theorem]{Proposition}

\theoremstyle{definition}
\newtheorem*{definition}{Definition}
\newtheorem*{remark}{Remark}

\numberwithin{equation}{section}

\renewcommand{\pmod}[1]{\ \left( \mathrm{mod} \, #1 \right)}

\setlist[enumerate]{leftmargin=*,label=\rm{(\arabic*)}}

\makeatletter
\@namedef{subjclassname@2020}{%
	\textup{2020} Mathematics Subject Classification}
\makeatother

\allowdisplaybreaks

\title[Hecke sum]{Proof of a conjecture of Andrews and Bachraoui on a Hecke sum}
\author{Koustav Banerjee}
\author{Kathrin Bringmann}
\address{University of Cologne, Department of Mathematics and Computer Science, Weyertal 86-90, 50931 Cologne, Germany}
\email{kbanerj1@uni-koeln.de}
\email{kbringma@uni-koeln.de}

\makeatletter
\@namedef{subjclassname@2020}{%
	\textup{2020} Mathematics Subject Classification}
\makeatother

\subjclass[2020]{11F11, 11F20, 11F37, 11P81}
\keywords{Hecke sum, indefinite theta functions, mock theta functions, modular forms, $q$-series, Sturm bounds, theta functions}

\begin{document}

\begin{abstract}

In this paper, we prove a conjecture of Andrews and Bachraoui relating a generating function arising from two-color partitions (with odd smallest part and restrictions on the even parts) to a Hecke-type double sum. Our proof is based on Zwegers' theory of indefinite theta functions together with modular transformation properties of mock theta functions.

\end{abstract}

\maketitle

\section{Introduction and statement of results}

 Andrews and Bachraoui \cite{AB} studied two-color partitions $C(k,n)$ ($k,n\in \mathbb{N}$). To be more precise, $C(k,n)$ counts the number of two-color partitions of $n\in \mathbb{N}$, with the following conditions:
 \begin{enumerate}
 	\item the smallest part is odd and occurs at least once in blue,
 	\item every even blue part is at least $2k-1$ greater than the smallest part,
 	\item the even parts of the same color are distinct.
 \end{enumerate}
 Andrews--Bachraoui \cite[(1.1)]{AB} showed that
\[
C_k(q):=\sum_{n\ge 0}C\left(k,n\right)q^n=\sum_{n\ge 0}\frac{\left(-q^{2n+2k},-q^{2n+2};q^2\right)_{\infty}q^{2n+1}}{\left(q^{2n+1};q^2\right)^2_{\infty}}.
\]  
Here, for $n\in \mathbb{N}_0\cup \{\infty\}$, $(a)_n=(a;q)_n:=\prod_{j=0}^{n-1}(1-aq^j)$ denotes the usual {\it $q$-Pochhammer symbol}. Andrews and Bachraoui \cite[Section 10]{AB} showed that as $k\to \infty$, the generating function $C_k(q)$ converges to a function
\[
C(q)=\lim_{k\to \infty}C_k(q)=\sum_{n\ge 0}\frac{\left(-q^{2n+2};q^2\right)_{\infty}q^{2n+1}}{\left(q^{2n+1};q^2\right)^2_{\infty}}.
\]
They then represented $C(q)$ as  (using different notation)
\begin{equation*}
C(q)=qD(q)S(q),
\end{equation*}
where
\begin{align*}
D(q)&:=\!\frac{\left(-q^2;q^2\right)_{\infty}}{\left(q;q^2\right)^2_{\infty}},\quad 
S(q):=\!\sum_{n\ge 0}\frac{\left(q;q^2\right)^2_n q^{2n}}{\left(-q^2;q^2\right)_n}.
\end{align*}
The function $S$ is the main object in this paper.

To isolate the even coefficients of $S$ we use a sieving operator. To be more precise, for $N\in \mathbb{N}$ and $r\in \mathbb{Z}$, the {\it sieving operator} $S_{N,r}$ is defined on a $q$-series $f(q)=\sum_{n\ge 0} c(n)q^n$ by
\[
f(q)\!\bigm|\!S_{N,r}:=\sum_{n\equiv r\pmod N}c(n) q^n,
\]
Andrews and Bachraoui \cite[Conjecture 4]{AB} proposed the following conjecture, relating a sieved generating function to a Hecke-type double sum.
\begin{conjecture}\label{ABConj3}
	For $N\in \mathbb{N}_0$, we have 
	\begin{align*}
	\left(\left(q^4;q^4\right)_{\infty}S(q)\right)\bigm| S_{2,0}=\sum_{n\ge 0}(-1)^nq^{6n^2+4n}\left(1+q^{4n+2}\right)\sum_{|j|\le n}(-1)^{j}q^{-2j^2}.
	\end{align*}
\end{conjecture}

In this paper, we prove Conjecture \ref{ABConj3}. The proof proceeds by ``completing" both sides of Conjecture \ref{ABConj3} to real-analytic modular forms, using Zwegers' theory of indefinite theta functions \cite{Zw}. We then show that these extra terms cancel and finally use Sturm's theorem. 
\begin{theorem}\label{BBthm}
The identity in Conjecture \ref{ABConj3} holds.
\end{theorem}


The paper is organized as follows. In Section \ref{sec1} we recall the necessary background on modular forms, mock theta functions, and indefinite theta functions. Section \ref{sec2} relates the right-hand side of Section \ref{ABConj3} to an indefinite theta function. In Section \ref{sec3}, we rewrite the left-hand side of Conjecture \ref{ABConj3}. Section \ref{sec4} establishes the required modular transformation properties. In Section \ref{sec5}, we complete the proof of Theorem \ref{BBthm}, using Sturm's theorem. Finally, in Section \ref{sec6}, we give a few questions for future research.

\section*{Acknowledgements}
The authors have received funding from the European Research Council (ERC) under the European Union's Horizon 2020 research and innovation programme (grant agreement No. 101001179).

\section{Preliminaries}\label{sec1}

\subsection{Modular forms} We briefly introduce modular forms, for more background, we refer the reader to \cite{Ono}. For $N\in \mathbb{N}$, the congruence subgroup $\G_0(N)$ is defined by
 \[
 \G_0\left(N\right):=\left\{\begin{pmatrix}
 a&b\\
 c&d
 \end{pmatrix}\in \textnormal{SL}_2(\mathbb{Z}): c\equiv 0\pmod{N}\right\}.
 \]
 \begin{definition}
 	{\it A holomorphic function $f:\mathbb{H}\rightarrow \mathbb{C}$ is a {\it modular form} of weight $k\in \mathbb{Z}$ for $\G_0(N)$ if it satisfies the {\it modular transformation property} 
 	\begin{equation}\label{transform}
 	f\left(\frac{a\t+b}{c\t+d}\right)=\left(c\t+d\right)^k f(\t)\ \ \text{for all}\ \begin{pmatrix}
 	a&b\\
 	c&d
 	\end{pmatrix}\in \G_0(N).
 	\end{equation}
 Moreover, it is required to have at most polynomial growth at the ``cusps" of $\G_0(N)$. We denote the complex vector space of modular forms of weight $k$ over $\G_0(N)$ by $M_k(\G_0(N))$.}
 \end{definition}

\begin{remark}
{\it Although some of the functions occuring in this paper have natural ``completions" to so-called harmonic Maass forms, we do not need this general formulation. We only require their explicit non-holomorphic completions and transformation properties.}
\end{remark}

We also use classical functions of half-integral weight, but only via their explicit transformations listed below. The {\it Dedekind eta-function} $\eta$ is defined by $(q=e^{2\pi i\t})$ \[\eta\left(\tau\right):=q^{\frac{1}{24}}\prod_{n\ge 1}\left(1-q^n\right).\]
The modular transformations 
of $\eta$ are given by (see \cite[Theorem 1.61]{Ono})
\begin{equation}\label{eta}
\eta(\t+1)=\z_{24}\eta(\t),\qquad \eta\left(-\frac{1}{\t}\right)=\sqrt{-i\t}\eta(\t),
\end{equation}
where $\z_N:=e^{\frac{2\pi i}{N}}$. Thus $\eta$ transforms like a modular form of weight $\frac 12$ (with a multiplier).

Next, define the {\it theta function}
\begin{equation}\label{Theta}
\Theta(\t):=\sum_{n\in \mathbb{Z}}q^{n^2}.
\end{equation}
We use the following transformation properties of $\Theta$
\begin{align}\label{theta0}
\Theta(\t+1)=\Theta(\t), \quad \Theta\left(-\frac{1}{2\t}\right)&=\sqrt{-i\t}\Theta\left(\frac{\t}{2}\right).
\end{align}
Again $\Theta$ transforms like a modular form of weight of $\frac 12$ on $\G_0(2)$ (with a multiplier).

We also require the following representation of $\Theta$ in terms of $\eta$ (see \cite[Theorem 1.60]{Ono})
\begin{align}\label{ONO}
\Theta(\t)=\frac{\eta^5(2\t)}{\eta^2(\t)\eta^2(4\t)},\quad \Theta\left(\t+\frac 12\right)=\frac{\eta^2(\t)}{\eta(2\t)},\quad \frac{\eta^2\left(16\t\right)}{\eta\left(8\t\right)}=\sum_{n\ge 0}q^{(2n+1)^2}.
\end{align}



\subsection{Valence formula} In this subsection, we recall two results that allow us to reduce the proof of identities to the computation of finitely many Fourier coefficients. In order to show identities between modular forms, we use the following consequence of the valence formula. First, we need the following result (see \cite[Proposition 1.7]{Ono}) on the index of $\G_0(N)$ in $\textnormal{SL}_2(\mathbb{Z})$.
\begin{proposition}\label{Onoprop}
	Let $N\in \mathbb{N}$. Then
	\[
	\left[\textnormal{SL}_2\left(\mathbb{Z}\right):\G_0(N)\right]=N\prod_{p\mid N}\left(1+\frac 1p\right).
	\]
\end{proposition}

Next, we state Sturm's result \cite{St}.
\begin{theorem}\label{Sturm}
	Let $k\in \mathbb{Z}, N\in \mathbb{N}$, and $f(\t)=\sum_{n\ge 0}c(n)q^n\in M_k(\G_0(N))$. If $c(n)=0$ for all $0\le n\le M$, where
	\[
	M\ge \frac{k}{12}\left[\textnormal{SL}_2\left(\mathbb{Z}\right):\G_0(N)\right],
	\]
	then $f= 0$.
\end{theorem}


\subsection{Mock theta functions} In this subsection we recall the mock theta function $\o(q)$ and its modular completion due to Zwegers which is used in this paper. Mock theta functions were introduced by Ramanujan and later understood in the framework of harmonic Maass forms.

Recall in particular Ramanujan's third order mock theta function $\o(q)$ and McIntosh's \cite{McIntosh} second order mock theta function $B(q)$ 
\begin{equation*}
\omega(q):=\sum_{n\ge0}\frac{q^{2n(n+1)}}{\left(q;q^2\right)^2_{n+1}},\qquad B(q):=\sum_{n\ge 0}\frac{\left(-q^2;q^2\right)_n q^{n(n+1)}}{\left(q;q^2\right)^2_{n+1}}.
\end{equation*}
Bachraoui and the authors \cite[Theorem 1.1]{BBB} related $S(q)$ to these mock theta functions.\footnote{Note that this identity is used in Section \ref{sec3}, to rewrite the left-hand side of Conjecture \ref{ABConj3}.}
\begin{theorem}\label{BBBthm}
	We have
	\[
	S(q)=2B(-q)-\frac{\left(q;q^2\right)^2_{\infty}}{\left(-q^2;q^2\right)_{\infty}}\o(-q).
	\]
\end{theorem}

Moreover, we require Ramanujan's third order mock theta function $f(q)$, which is defined by
\begin{equation*}
f(q):=\sum_{n\ge0}\frac{q^{n^2}}{\left(-q\right)^2_{n}}.
\end{equation*}

 We use the vector-valued\footnote{Throughout we write vectors in bold and their components in non-bold letters with indices as is done here for $\bm{F}$.} functions introduced by Zwegers\footnote{Note that Zwegers used slightly different notation.} \cite{Zwegers}
\begin{align}\label{defF}
\bm{F}\left(\t\right)&=\left(F_1(\t),F_2(\t),F_3(\t)\right)^T:=\left(q^{-\frac{1}{24}}f(q),2q^{\frac 13}\o\!\left(q^{\frac 12}\right), 2q^{\frac 13}\o\!\left(-q^{\frac 12}\right)\right)^T,\nonumber\\
\bm{F}^{-}\left(\t\right)\!&=\!\left(F^{-}_1(\t),F^{-}_2(\t), F^{-}_3(\t)\right)^T\!\nonumber\\
&:=\!-2i\sqrt{3}\int^{i\infty}_{-\overline{\t}}\frac{\left(t_1(w),t_2(w),t_3(w)\right)^T}{\sqrt{-i\left(w+\t\right)}}dw,
\end{align}
where the theta functions $t_1$, $t_2$, and $t_3$ are defined by
\begin{align*}
t_1(\t)&:=-\sum_{n\in \mathbb{Z}}\left(n+\frac 16\right)e^{3\pi i\left(n+\frac 16\right)^2\t},\quad t_2(\t):=\sum_{n\in \mathbb{Z}}(-1)^n \left(n+\frac 13\right)e^{3\pi i\left(n+\frac 13\right)^2\t},\nonumber\\
t_3(\t)&:=-\sum_{n\in \mathbb{Z}}\left(n+\frac 13\right)e^{3\pi i\left(n+\frac 13\right)^2\t}.
\end{align*}
Letting
\[
\widehat{\bm{F}}(\t):=\left(\widehat{F}_1(\t), \widehat{F}_2(\t), \widehat{F}_3(\t)\right)^T,
\]
Zwegers' proved the following theorem \cite{Zwegers} that is central for our arguments as it provides a completion for $\o(q)$. 
\begin{theorem}\label{Zthm}
	The function $\widehat{\bm{F}}$ is a vector-valued real-analytic modular form of weight $\frac 12$ that satisfies the following transformations:
	\begin{align*}
	\widehat{\bm{F}}(\t+1)&=\begin{pmatrix}
	\z^{-1}_{24}&0 &0\\
	0 &0 & \z_3\\
	0&\z_3&0
	\end{pmatrix}\widehat{\bm{F}}(\t),\qquad \widehat{\bm{F}}\left(-\frac{1}{\t}\right)=\sqrt{-i\t}\begin{pmatrix}
	0&1 &0\\
	1 &0 & 0\\
	0&0&-1
	\end{pmatrix}\widehat{\bm{F}}(\t).
	\end{align*}
\end{theorem}

\subsection{Indefinite theta functions} In this subsection, we recall Zwegers' construction \cite{Zw} of indefinite theta functions which is used below. Let $n\in \mathbb{N}$ and fix a quadratic form $Q$ on $\mathbb{R}^{\ell}$ ($\ell\in \mathbb{N}$) of signature $(n,1)$ with associated matrix $A$, so that  $Q(\bm{x})=\frac 12\bm{x}^TA\bm{x}$. In this paper we have $\ell=2$ and $n=1$. Let $B(\bm{x},\bm{y}):=Q(\bm{x}+\bm{y})-Q(\bm{x})-Q(\bm{y})$ denote the corresponding bilinear form. Note that $B(\bm{x},\bm{x})=2Q(\bm{x})$. The set of vectors $\bm{c}\in \mathbb{R}^{\ell}$ with $Q(\bm{c})<0$ splits into two connected  components. Two vectors $\bm{c_1}$ and $\bm{c_2}$ lie in the same component iff $B(\bm{c_1},\bm{c_2})<0$. We fix one of the components and denote it by $C_{Q}$. Picking any vector $\bm{c_0}\in C_Q$, we have
\begin{equation}\label{CQ}
C_Q=\left\{\bm{c}\in \mathbb{R}^{\ell}:Q(\bm{c})<0, B\left(\bm{c},\bm{c_0}\right)<0\right\}.
\end{equation}
Consider (the possibly empty set)
\[
S_Q:=\left\{\bm{c}\in \mathbb{Z}^{\ell}:\gcd\left(c_1,c_2,\cdots,c_{\ell}\right)=1,\ Q(\bm{c})=0,\ B\left(\bm{c},\bm{c_0}\right)<0\right\}.
\]
Let $\overline{C}_Q:=C_Q\cup S_Q$ and define for, $c\in \overline{C}_Q$,
\[
R(\bm{c}):=\begin{cases}
\mathbb{R}^{\ell} &\text{if}\ \bm{c}\in C_Q,\\
\left\{\bm{a}\in \mathbb{R}^{\ell}:B\left(\bm{c},\bm{a}\right)\notin\mathbb{Z}\right\} &\text{if}\ \bm{c}\in S_Q.
\end{cases}
\]
For $\bm{c_1},\bm{c_2}\in \overline{C}_Q$, we define the {\it theta function with characteristic} $\bm{a}\in R(\bm{c_1})\cap R(\bm{c_2})$ and $\bm{b}\in \mathbb{R}^{\ell}$ by\footnote{We use the notation $\vartheta_{\bm{a},\bm{b}}(\t)$ if $\bm{c_1}$ and $\bm{c_2}$ are fixed.} 
\begin{equation}\label{theta}
\vartheta_{\bm{a},\bm{b}}(\t)=\vartheta^{\bm{c_1},\bm{c_2}}_{\bm{a},\bm{b}}(\t)
:=\sum_{\bm{n}\in \mathbb{Z}^{\ell}+\bm{a}}\varrho\left(\bm{n};\t\right)e^{2\pi iB\left(\bm{b},\bm{n}\right)}q^{Q(\bm{n})},
\end{equation}
where
\[
\varrho\left(\bm{n};\t\right):=\varrho^{\bm{c_1}}(\bm{n};\t)-\varrho^{\bm{c_2}}(\bm{n};\t)
\]
with ($\t=u+iv$)
\[
\varrho^{\bm{c}}(\bm{n};\t):=\begin{cases}
E\left(\frac{B\left(\bm{c},\bm{n}\right)\sqrt{v}}{\sqrt{-Q(\bm{c})}}\right) &\text{if}\ c\in C_Q,\\
\textnormal{sgn}\left(B(\bm{c},\bm{n})\right) &\text{if}\ c\in S_Q.
\end{cases}
\]
Here the odd entire function $E$ is defined as 
\[
E(w):=2\int_{0}^{w}e^{-\pi t^2}dt
\]
with the usual convention that $\textnormal{sgn}(w):=\frac{w}{|w|}$ for $w\in \mathbb{R}\setminus\{0\}$ and $\textnormal{sgn}(0):=0$. Note that (see \cite[Lemma 1.7]{Zw}) 
\begin{equation}\label{Ebeta}
E(x)=\textnormal{sgn}(x)\left(1-\b\left(x^2\right)\right),\ \text{with}\ \b(x)\!:=\!\!\int_{u}^{\infty}\!\!w^{-\frac 12}e^{-\pi w}dw=\frac{\G\left(\frac 12,\pi u\right)}{\sqrt{\pi}},
\end{equation}
where $\G(a,v):=\int_{v}^{\infty}e^{-t}t^{a-1}dt$, for $v\in \mathbb{R}^{+}$, is the {\it incomplete gamma function}.\\
Due to Zwegers' \cite[Corollary 2.9]{Zw}, $\vartheta_{\bm{a},\bm{b}}$ satisfies the following transformations:
\begin{theorem}\label{Zwmainthm1}
	The function $\vartheta_{\bm{a},\bm{b}}$ has the following elliptic and modular transformation properties:
	\begin{enumerate}
		\item $\vartheta_{\bm{a}+\bm{\l},\bm{b}}=\vartheta_{\bm{a},\bm{b}}$ for $\bm{\l}\in \mathbb{Z}^{\ell}$,
		\item $\vartheta_{\bm{a},\bm{b}+{\bm \mu}}=e^{2\pi iB(\bm{a},\bm{\mu})}\vartheta_{\bm{a},\bm{b}}$ for $\bm{\mu}\in A^{-1}\mathbb{Z}^{\ell}$,
		\item $\vartheta_{-\bm{a},-\bm{b}}=-\vartheta_{\bm{a},\bm{b}}$,
		\item $\vartheta_{\bm{a},\bm{b}}(\t+1)=e^{-2\pi i Q(\bm{a})-\pi i B(A^{-1}A^*,\bm{a})}\vartheta_{\bm{a},\bm{b}+\bm{a}+\frac 12A^{-1}A^{*}}(\t)$ with $A^*$ denoting the vector of diagonal elements of $A$.
		\item If $\bm{a}, \bm{b}\in R(\bm{c_1})\cap R(\bm{c_2})$, then
		\[
		\vartheta_{\bm{a},\bm{b}}\left(-\frac 1\t\right)=\frac{i}{\sqrt{-\det (A)}}\left(-i\t\right)^{\frac{\ell}{2}}e^{2\pi i B(\bm{a},\bm{b})}\sum_{\bm{\nu}\in A^{-1}\mathbb{Z}^{\ell}\pmod{\mathbb{Z}^{\ell}}}\vartheta_{\bm{b}+\bm{\nu},-\bm{a}}(\t).
		\]
	\end{enumerate}
\end{theorem}

Moreover, we also require transformation of $\vartheta_{\bm{a},\bm{b}}$ with respect to a subgroup of $\text{GL}_2(\mathbb{R})$. Consider the subgroup of transformations preserving the quadratic form
\[
O_A(\mathbb{R}):=\left\{C\in \text{GL}_{\ell}\left(\mathbb{R}\right): C^T AC=A\right\}.
\]  
We then restrict to those matrices $C$ that leave $C_Q$ invariant, i.e., $B(C\bm{c}, \bm{c})<0$ for $\bm{c}\in C_Q$. Set\footnote{By \eqref{CQ} we only need to show that $B(C\bm{c},\bm{c})<0$ for one $c\in C_Q$.}
\[
O^{+}_A(\mathbb{R}):=\left\{C\in\text{GL}_{\ell}(\mathbb{R}): C^TAC=A,\ B\left(C\bm{c},\bm{c}\right)<0\ \text{for all}\ c\in C_Q \right\}.
\]
and (see \cite[Definition 2.10]{Zw})
\[
O^{+}_A(\mathbb{Z}):=O^{+}_A(\mathbb{R})\cap \text{GL}_{\ell}(\mathbb{R}).
\] 
From \cite[Corollary 2.15]{Zw}, we have the following result of Zwegers.
\begin{theorem}\label{Zwmainthm2}
	Let $C\in O^{+}_A(\mathbb{Z})$, $\bm{c_1}, \bm{c_2}\in \overline{C}_Q$, and $\bm{a}\in R(\bm{c_1})\cap R(\bm{c_2})$. Then
	\[
	\vartheta^{C\bm{c_1}, C\bm{c_2}}_{C\bm{a},C\bm{b}}(\t)=\vartheta^{\bm{c_1}, \bm{c_2}}_{\bm{a},\bm{b}}(\t).
	\]
\end{theorem}

\section{An indefinite theta function}\label{sec2}

In this section, we write the right-hand side of Conjecture \ref{ABConj3} as the holomorphic part of an indefinite theta function and determine its non-holomorphic contribution. Replacing $q\mapsto q^{\frac 12}$, we set 
\begin{equation*}
H(q):=\sum_{n\ge 0}(-1)^nq^{3n^2+2n}\left(1+q^{2n+1}\right)\sum_{|j|\le n}(-1)^{j}q^{-j^2}.
\end{equation*}


Define 
\begin{align}\nonumber
H^{-}(\t)&:=-\frac{1}{\sqrt{\pi}}\sum_{\varepsilon\in\{0,1\}}\!\!(-1)^{\varepsilon}\sum_{\substack{n\in \mathbb{Z}\\r\in \mathbb{Z}}}\text{sgn}\left(2r+\varepsilon+\frac 13\right)\\\label{hhat1}
&\hspace{1 cm}\times\Gamma\left(\frac 12,6\pi\left(2r+\varepsilon+\frac 13\right)^2v\right)q^{-\frac 32\left(2r+\varepsilon+\frac 13\right)^2+\frac 12\left(2n+\varepsilon+1\right)^2}.
\end{align}
We then set
\begin{equation}\label{hhat2}
\widehat{H}(\t):=q^{\frac 13}H(q)+H^{-}(\t).
\end{equation}

We next realize $\widehat{H}$ as an indefinite theta function of Zwegers. For this, let 
\begin{align*}
\bm{a}&:=\left(\frac 13,0\right)^T,\quad \bm{b}:=\left(\frac{1}{12},-\frac 14\right)^T
\end{align*}
and define the quadratic form
\[
Q(\bm{n}):=3n^2_1-n^2_2.
\]
We choose vectors $\bm{c_1}$ and $\bm{c_2}$ as
\[
\bm{c_1}=(1,3)^T,\quad \bm{c_2}=(-1,3)^T
\]
and fix the component $C_{Q}$ so that $\bm{c_1}\in C_Q$. Then also $\bm{c_2}\in C_Q$.

\begin{proposition}\label{prop1}
We have
\[
\widehat{H}(\t)=\frac{e^{-\frac{\pi i}{3}}}{2}\vartheta_{\bm{a},\bm{b}}(\t).
\]
\end{proposition}

\begin{proof}
	
We may rewrite $H(q)$ as follows:  
\begin{equation*}
H(q)=\sum_{\substack{n+j\ge 0\\n-j\ge 0}}(-1)^{n+j}q^{3n^2-j^2+2n}+\sum_{\substack{n+j\ge 0\\n-j\ge 0}}(-1)^{n+j}q^{3n^2-j^2+4n+1}.
\end{equation*}
  Making in the second sum the change of variables $n\mapsto -n-1, j\mapsto -j$, we obtain (noting the sign-change in the second summand) 
\begin{equation*}
H(q)=\left(\sum_{\substack{n+j\ge 0\\n-j\ge 0}}-\sum_{\substack{n+j< 0\\n-j< 0}}\right)(-1)^{n+j}q^{3n^2-j^2+2n}.
\end{equation*}
Let $n_1:=n+\frac 13$ and $n_2:=j$. Using $3n^2+2n=3(n+\frac13)^2-\frac13$ and $(-1)^n=e^{-\frac{\pi i}{3}}e^{\pi in_1}$, we obtain  
\begin{equation*}
H(q)=\frac{e^{-\frac{\pi i}{3}}q^{-\frac 13}}{2}\!\sum_{\bm{n}\in \mathbb{Z}^2+\bm{a}} \left(\text{sgn}\left(n_1+n_2\right)+\text{sgn}\left(n_1-n_2\right)\right)e^{\pi i\left(n_1+n_2\right)}q^{3n^2_1-n^2_2}.
\end{equation*}
Using $B(\bm{c_1},\bm{n})=6(n_1-n_2)$, $B(\bm{c_2},\bm{n})=-6(n_1+n_2)$, and $B(\bm{b},\bm{n})=\frac{n_1+n_2}{2}$, this becomes  
\[
H(q)=\frac{e^{-\frac{\pi i}{3}}q^{-\frac 13}}{2}\!\sum_{\bm{n}\in \mathbb{Z}^2+\bm{a}}\left(\text{sgn}\left(B\left(\bm{c_1},\bm{n}\right)\right)-\text{sgn}\left(B\left(\bm{c_2},\bm{n}\right)\right)\right)e^{2\pi i B\left(\bm{b}, \bm{n}\right)}q^{Q\left(\bm{n}\right)}.
\]
Thus, by \eqref{Ebeta}, \eqref{theta}, and the definitions of $\bm{c_j},\, Q(\bm{n})$, and  $B(\bm{n},\bm{m})$, we have (reverting the above calculations) 
\begin{align*}
H^*(\t)&:=\frac{e^{-\frac{\pi i}{3}}}{2}\vartheta_{\bm{a},\bm{b}}(\t)-q^{\frac 13}H(q)\\
&\hspace{0.1 cm}=-\frac{q^{\frac 13}}2\!\sum_{n,j\in \mathbb{Z}}\Bigg(\textnormal{sgn}\!\left(n\!+\!j\!+\!\frac 13\right)\!\b\left(\!6\left(n\!+\!j\!+\!\frac 13\right)^2v\right)\\
&\hspace{2 cm}+\!\textnormal{sgn}\!\left(n\!-\!j\!+\!\frac 13\right)\!\b\left(6\!\left(n\!-\!j\!+\!\frac 13\right)^2 v\right)\Bigg)(-1)^{n+j}q^{3n^2+2n-j^2}.
\end{align*}

To finish the proof, it remains to show that
\begin{equation}\label{Hsum}
H^*(\t)=H^{-}(\t).
\end{equation}
Now letting $j\mapsto -j$ in the second summand and then setting $r=n+j$, we obtain the following expression
\begin{align}\label{ref6}
H^{*}(\t)
&=-\sum_{n,r\in \mathbb{Z}}\text{sgn}\left(r+\frac 13\right)(-1)^{r}\b\left(6\left(r+\frac 13\right)^2 v\right)\nonumber\\
&\hspace{6 cm}\times q^{2\,\left(n+\frac{r+1}{2}\right)^2-\frac 32\left(r+\frac 13\right)^2}.
\end{align}
Splitting the terms according to the parity, i.e., $r\mapsto 2r+\varepsilon$ ($\varepsilon\in\{0,1\}$) and then $n\mapsto n-r$ in \eqref{ref6}, we obtain 
\begin{align*}
H^*(\t)&=-
\sum_{\varepsilon\in\{0,1\}}(-1)^{\varepsilon}\sum_{n\in \mathbb{Z}}q^{2\,\left(n+\frac{\varepsilon+1}{2}\right)^2}\\
&\hspace{2 cm}\times \sum_{r\in \mathbb{Z}}\text{sgn}\left(2r\!+\!\varepsilon\!+\!\frac 13\right)\b\left(6\!\left(2r\!+\!\varepsilon\!+\!\frac 13\right)^2\! v\right)q^{-\frac 32\left(2r\!+\varepsilon+\!\frac 13\right)^2}\\
&=-\frac{1}{\sqrt{\pi}}\sum_{\varepsilon\in\{0,1\}}(-1)^{\varepsilon}\sum_{n\in \mathbb{Z}}q^{\frac 12\left(2n+\varepsilon+1\right)^2}\\
&\hspace{1 cm}\times\sum_{r\in \mathbb{Z}}\text{sgn}\left(2r+\varepsilon+\frac 13\right) \Gamma\left(\frac 12,6\pi\left(2r+\varepsilon+\frac 13\right)^2v\right)q^{-\frac 32\left(2r+\varepsilon+\frac 13\right)^2},
\end{align*}
using \eqref{Ebeta}. Comparing with the definition of $H^{-}(\t)$ gives \eqref{Hsum}.\qedhere 
\end{proof}

\section{Completing the function $(q^4;q^4)_{\infty}S(q)|S_{2,0}$}\label{sec3}

In this section, we relate the left-hand side of Conjecture \ref{ABConj3} to explicit combinations of mock theta functions. For this, define 
\begin{align*}
A(q)&:=\left(q^2;q^2\right)_{\infty}\left[S(q)\Bigm|S_{2,0}\right]_{q\mapsto q^{\frac 12}}.
\end{align*}
Thus $A(q)$ is a rescaled version of the left-hand side of Conjecture \ref{ABConj3}. Moreover define the non-holomorphic part by
\begin{align*} 
A^{-}(\t)&:=-\frac{1}{\sqrt{\pi}}\sum_{\varepsilon\in\{0,1\}}\!\!(-1)^{\varepsilon}\sum_{\substack{n\in \mathbb{Z}\\r\in \mathbb{Z}}}\text{sgn}\left(2r+\varepsilon+\frac 13\right)\\
&\hspace{2 cm}\times\Gamma\left(\frac 12,6\pi\left(2r+\varepsilon+\frac 13\right)^2v\right)q^{-\frac 32\left(2r+\varepsilon+\frac 13\right)^2+\frac 12\left(2n+\varepsilon+1\right)^2}.
\end{align*}
With this, we then let
\begin{align*}
\widehat{A}(\t)&:=q^{\frac 13}A\!\left(q\right)+A^{-}(\t).
\end{align*}

\begin{proposition}\label{prop2}
We have the identity
\begin{align*}
\widehat{A}(\t)=2\frac{\eta^6(2\t)}{\eta^4(\t)}-\frac 14\left(\Theta\!\left(\frac{\t}{2}\right)\widehat{F}_2(\t)+\Theta\!\left(\frac{\t+1}{2}\right)\widehat{F}_3\left(\t\right)\right).
\end{align*}
\end{proposition}
\begin{proof} 
We now use Theorem \ref{BBBthm} to rewrite $A$ in terms of mock theta functions. Namely we have 
\[
A\left(q\right)=\left(q^2;q^2\right)_{\infty}\left[\left(2B(-q)-\frac{\left(q;q^2\right)^2_{\infty}}{\left(-q^2;q^2\right)_{\infty}}\o(-q)\right)\Biggm|S_{2,0}\right]_{q\mapsto q^{\frac 12}}.
\]
Now note that for a $q$-series
\[
g(-q)|S_{2,0}=g(q)|S_{2,0},
\]
since only even powers of $q$ survive when applying $S_{2,0}$. Thus, we have  
\begin{align}\label{splitF}
A(q)
&=\left(q^2;q^2\right)_{\infty}\left(2A_1(q)-\frac{A_2(q)}{\left(-q\right)_{\infty}}\right),
\end{align}
where
\[
A_1(q):=\left[B(q)\bigm|S_{2,0}\right]_{q\mapsto q^{\frac 12}},\quad A_2(q):=\left[\left(\left(-q;q^2\right)^2_{\infty}\o(q)\right)\Bigm|S_{2,0}\right]_{q\mapsto q^{\frac 12}}.
\]

Now, by the example below Theorem 1.3 of \cite{BOR}, we have
\begin{align}\label{firstpart}
A_1(q)
&=\frac{\left(q^2;q^2\right)^5_{\infty}}{\left(q\right)^4_{\infty}}.
\end{align}
Moreover, by the first identity in \eqref{ONO}, we may write 
\begin{equation*}
A_2(q)=\frac{1}{\left(q\right)_{\infty}}\left[\left(\Theta(\t)\o(q)\right)\Bigm|S_{2,0}\right]_{q\mapsto q^{\frac 12}}.
\end{equation*}
Plugging this and \eqref{firstpart} into \eqref{splitF}, we have  
\begin{equation}\label{eqn1}
A(q)=2q^{-\frac 13}\frac{\eta^6\left(2\t\right)}{\eta^4(\t)}-\left[\left(\Theta(\t)\o(q)\right)\Bigm|S_{2,0}\right]_{q\mapsto q^{\frac 12}}.
\end{equation}

We next rewrite the second summand in \eqref{eqn1}. For this, let
\[
g(\t):=\Theta(\t)\o(q),\quad G\left(\t\right):=\left[g(\t)\Bigm|S_{2,0}\right]_{q\mapsto q^{\frac 12}}. 
\] 
Plugging this into \eqref{eqn1}, we obtain 
\begin{equation*}
q^{\frac 13}A\!\left(q\right)=2\frac{\eta^6\left(2\t\right)}{\eta^4(\t)}-q^{\frac 13}G(\t).
\end{equation*}
Note that 
\begin{align*}
G\left(\t\right)
&=\frac 12\left(g\left(\frac{\t}{2}\right)+g\left(\frac{\t+1}{2}\right)\right)\nonumber\\
&=\frac 12\left(\Theta\left(\frac{\t}{2}\right)\omega\Bigl(q^{\frac{1}{2}}\Bigr)+\Theta\left(\frac{\t+1}{2}\right)\omega\left(-q^{\frac 12}\right)\right)\nonumber\\
&=\frac{q^{-\frac 13}}{4}\left(\Theta\left(\frac{\t}{2}\right)F_2(\t)+\Theta\left(\frac{\t+1}{2}\right)F_3(\t)\right).
\end{align*}

It remains to show that
\begin{equation}\label{ident}
A^*(\t)=A^{-}(\t),
\end{equation}
where
\begin{equation}\label{Astar}
A^*(\t):=-\frac{1}{4}\left(\Theta\left(\frac{\t}{2}\right)F^{-}_2(\t)+\Theta\left(\frac{\t+1}{2}\right)F^{-}_3(\t)\right).
\end{equation}
For this, we first rewrite $F^{-}_2$ and $F^{-}_3$. By definition of $t_2$ (see below \eqref{defF}), we have
\begin{equation}\label{ref10}
\int_{-\overline{\t}}^{i\infty}\frac{t_2(w)}{\sqrt{-i(w+\t)}}dw=\sum_{n\in \mathbb{Z}}(-1)^n\left(n+\frac 13\right)\int_{-\overline{\t}}^{i\infty}\frac{e^{3\pi i\left(n+\frac 13\right)^2 w}}{\sqrt{-i(w+\t)}}dw.
\end{equation}
Now, making the change of variables $w\mapsto \frac{it}{3\pi(n+\frac 13)^2}-\t$, we compute 
\begin{align*}
\int_{-\overline{\t}}^{i\infty}\frac{e^{3\pi i\left(n+\frac 13\right)^2 w}}{\sqrt{-i(w+\t)}}dw
&=\frac{i}{\sqrt{3\pi}\left|n+\frac 13\right|}q^{-\frac 32\left(n+\frac 13\right)^2}\int_{6\pi\left(n+\frac 13\right)^2v}^{\infty}\sqrt{t}e^{-t}\frac{dt}{t}\nonumber\\
&=\frac{i\Gamma\!\left(\frac 12,6\pi\left(n+\frac 13\right)^2v \right)q^{-\frac 32\left(n+\frac 13\right)^2}}{\sqrt{3\pi}\left|n+\frac 13\right|}.
\end{align*}
 Plugging this into \eqref{ref10} and then using the definition of $F^{-}_2$ in \eqref{defF}, we obtain
\[
F^{-}_2(\t)=\frac{2}{\sqrt{\pi}}\sum_{n\in \mathbb{Z}}(-1)^n\text{sgn}\left(n+\frac 13\right)\Gamma\!\left(\frac 12,6\pi\left(n+\frac 13\right)^2v \right)q^{-\frac 32\left(n+\frac 13\right)^2}.
\]
Analogously, we have  
\begin{align*}
F^{-}_3(\t)&=-\frac{2}{\sqrt{\pi}}\sum_{n\in \mathbb{Z}}\text{sgn}\left(n+\frac 13\right)\Gamma\!\left(\frac 12,6\pi\left(n+\frac 13\right)^2v \right)q^{-\frac 32\left(n+\frac 13\right)^2}.
\end{align*}
Thus, using \eqref{Theta} and then plugging into \eqref{Astar}, we obtain
\begin{align*}
A^*(\t)&=-\frac{1}{2\sqrt{\pi}}\sum_{\substack{n\in \mathbb{Z}\\m\in \mathbb{Z}}}\text{sgn}\left(n+\frac 13\right)\left((-1)^n-(-1)^m\right)\\[-0.75 em]
&\hspace{3 cm}\times \Gamma\!\left(\frac 12,6\pi\left(n+\frac 13\right)^2v \right)q^{\frac{m^2}{2}-\frac 32\left(n+\frac 13\right)^2}.
\end{align*}
 The factor $(-1)^n-(-1)^m$ vanishes unless $n\equiv m+1\pmod{2}$. Splitting the sum according to this parity condition gives exactly the expression defining $A^{-}(\t)$. Hence \eqref{ident} follows.\qedhere 
\end{proof}

Define
\[
\widehat{M}(\t):=\widehat{H}(\t)-\widehat{A}(\t).
\]
By Proposition \ref{prop1} and Proposition \ref{prop2}, $\widehat{M}(\t)$ is actually a $q$-series.


\begin{corollary}\label{cor1}
 We have
\[
\widehat{M}(\t)=q^{\frac 13}\left(H(q)-A(q)\right).
\]
\end{corollary}


\section{A modularity result}\label{sec4}

In this section we determine modularity properties of $\widehat{M}$. We start with $\widehat{A}$.

\begin{proposition}\label{lemma2}
The function $\widehat{A}$ satisfies\footnote{We could show that $\widehat{A}$ is a harmonic Maass form of weight $1$ with multiplier. However, we do not require this fact for the paper.}
\[
\widehat{A}(\t+1)=\z_3\widehat{A}(\t),\qquad \widehat{A}\left(\frac{\t}{2\t+1}\right)=\z_{12}(2\t+1)\widehat{A}(\t).
\]	
\end{proposition}
\begin{proof}
Using Proposition \ref{prop2}, we write
\begin{equation*}
\widehat{A}(\t):=2\mathcal{E}(\t)-\frac{\widehat{G}(\t)}{4},
\end{equation*}
where
\begin{equation}\label{mathcalE}
\mathcal{E}(\t):=\frac{\eta^6\left(2\t\right)}{\eta^4(\t)},\qquad \widehat{G}(\t):=\sum_{j\in \{0,1\}}\Theta\left(\frac{\t+j}{2}\right)\widehat{F}_{2+j}\left(\t\right).
\end{equation}
By a direct calculation using \eqref{eta}, we obtain that $\mathcal{E}$ satisfies the required transformations stated in the proposition.

Thus it remains to show that $\widehat{G}$ satisfies the same transformation properties. For this, we let
\begin{align*}
T_1(\t):=\Theta\left(\frac{\t}{2}\right),\quad T_2(\t):=\Theta\left(\frac{\t}{2}+\frac 12\right),\quad T_3(\t):=
2\sum_{n\ge 0}q^{\frac{(2n+1)^2}8}.
\end{align*}
Set
\[
\bm{T}(\t):=\left(T_1(\t), T_2(\t), T_3(\t)\right)^T.
\]
By \eqref{Theta}, \eqref{theta0}, \eqref{ONO}, and \eqref{eta}, we obtain
\begin{equation}\label{Ttrans}
\bm{T}(\t+1)=\begin{pmatrix}
0&1&0\\
1&0&0\\
0&0&\z_8 
\end{pmatrix}\bm{T}(\t), \quad \bm{T}\left(-\frac{1}{\t}\right)=\sqrt{-i\t}\begin{pmatrix}
1&0&0\\
0&0&1\\
0&1&0 
\end{pmatrix}\bm{T}(\t).
\end{equation}
Note that with $S:=(\begin{smallmatrix}
0&-1\\1&0
\end{smallmatrix})$, we have 
\[
S^{-1}T^{-2}S=\begin{pmatrix}
1&0\\
2&1
\end{pmatrix}.
\] Thus, by \eqref{Ttrans}, we may conclude that 
\begin{equation}\label{conc}
\bm{T}\left(\frac{\t}{2\t+1}\right)=\sqrt{2\t+1}\begin{pmatrix}
1&0&0\\
0&\z^{3}_4&0\\
0&0&1
\end{pmatrix}\bm{T}\left(\t\right).
\end{equation}
Similarly, using Theorem \ref{Zthm}, one obtains 
\[
\widehat{\bm{F}}\left(\frac{\t}{2\t+1}\right)=\sqrt{2\t+1}\begin{pmatrix}
\z_{3}&0 &0\\
0 &\z_{12} & 0\\
0&0&\z_3
\end{pmatrix}\widehat{\bm{F}}(\t).
\]
Combining this with \eqref{conc}, each summand in $\widehat{G}$ transforms with the same factor as in Proposition \ref{lemma2} which proves the claim.\qedhere 
\end{proof}




Next we show modular transformations for $\widehat{H}$. 

\begin{proposition}\label{hathtrans}
We have\footnote{Again we could prove that $\widehat{H}$ is a harmonic Maass form of weight $1$ with multiplier.}
\[
\widehat{H}\left(\t+1\right)=\z_3\widehat{H}(\t),\qquad \widehat{H}\left(\frac{\t}{2\t+1}\right)=\z_{12}(2\t+1)\widehat{H}(\t).
\]
\end{proposition}
\begin{proof}
By Theorem \ref{Zwmainthm1} (4), we have 
\[
\vartheta_{\bm{a},\bm{b}}(\t+1)=\z^{2}_3\vartheta_{\bm{a},\bm{b}+\left(\frac{5}{6},\frac 12\right)}(\t).
\]
Using Theorem \ref{Zwmainthm1} (2) and Proposition \ref{prop1} gives the first claim. Note that the shift is allowed since $(\frac 56,\frac 12)$ lies in $A^{-1}\mathbb{Z}^2$.


We next note that, by Theorem \ref{Zwmainthm1} (5)
\begin{align}\label{star}
\vartheta_{\bm{a},\bm{b}}\left(-\frac{1}{\t}\right)
&=\frac{\z_6\t}{2\sqrt{3}}\sum_{\nu_1,\nu_2}\vartheta_{\bm{b}+\bm{\nu},-\bm{a}}(\t),
\end{align}
where $\nu_1$ runs through a set of representatives of $(\frac 16\mathbb{Z})/\mathbb{Z}$ and $\nu_2$ runs through a set of representatives of $(\frac 12\mathbb{Z})/\mathbb{Z}$. We substitute $\nu_1=\frac{\ell_1}{6}$ and $\nu_2=\frac{\ell_2}{2}$ with $\ell_1$ running modulo $6$ and $\ell_2$ running modulo $2$. We then define
\begin{align}\label{l1l2}
\bm{\rho_{\ell_1,\ell_2}}&:=\bm{b}+\left(\frac{\ell_1}{6},\frac{\ell_2}{2}\right).
\end{align}
Now let $C:=(\begin{smallmatrix}
-1&0\\
\ 0&1
\end{smallmatrix})$. Note that  $C\in O^{+}_{A}(\mathbb{Z})$ as $C^TAC=A$ and $C\bm{c_1}=\bm{c_2}$ which lies in the same component as $\bm{c_1}$. Thus, by Theorem \ref{Zwmainthm2}, we have
\begin{equation*}
\vartheta^{\bm{c_1}, \bm{c_2}}_{\bm{\a},\bm{\b}}=\vartheta^{C\bm{c_1}, C\bm{c_2}}_{C\bm{\a},C\bm{\b}}.
\end{equation*}
Noting that $C\bm{c_1}=\bm{c_2}$, $C\bm{c_2}=\bm{c_1}$, we obtain, as interchanging $\bm{c_1}$ and $\bm{c_2}$ changes the sign of the theta function, 
\[
\vartheta_{\bm{\a},\bm{\b}}=-\vartheta_{C\bm{\a},C\bm{\b}}.
\]
Consequently, we obtain
\begin{equation}\label{plug1}
\vartheta_{\bm{\rho_{\ell_1,\ell_2}},-\bm{a}}=-\vartheta_{C\bm{\rho_{\ell_1,\ell_2}},C(-\bm{a})}=-\vartheta_{C\bm{\rho_{\ell_1,\ell_2}},\bm{a}},
\end{equation}
 using the fact that $C(-\bm{a})=\bm{a}$ in the final step. Now, by \eqref{l1l2}, we have 
\[
C\bm{\rho_{\ell_1,\ell_2}}\equiv \bm{\rho_{5-\ell_1,\ell_2}}\pmod{\mathbb{Z}^2}.
\]
Plugging this into \eqref{plug1} and using Theorem \ref{Zwmainthm1} (2) gives 
\[
\vartheta_{\bm{\rho_{\ell_1,\ell_2}},-\bm{a}}
=-\z^{\ell_1+2}_3\vartheta_{\bm{\rho_{5-\ell_1,\ell_2}},-\bm{a}}.
\]
 Inserting this into \eqref{star} yields 
\begin{align}\label{ref3}
\vartheta_{\bm{a},\bm{b}}\left(-\frac{1}{\t}\right)&=\frac{\z_6\t}{2\sqrt{3}}\sum_{\substack{\ell_1\in\{0,1,2\}\\\ell_2\in\{0,1\}}}\left(1-\z^{2\ell_1+1}_3\right)\vartheta_{\bm{\rho_{\ell_1,\ell_2}},-\bm{a}}(\t)\nonumber\\
&=\frac{\z_6\t}{2\sqrt{3}}\sum_{\substack{\ell_1\in\{0,2\}\\\ell_2\in\{0,1\}}}\left(1-\z^{2\ell_1+1}_3\right)\vartheta_{\bm{\rho_{\ell_1,\ell_2}},-\bm{a}}(\t),
\end{align}
as the term with $\ell_1=1$ vanishes.

 Next, by Theorem \ref{Zwmainthm1} (4), we obtain 
\begin{align*}
&\vartheta_{\bm{\rho_{\ell_1,\ell_2}},-\bm{a}}(\t-2)\nonumber\\
&\hspace{1.5cm}=e^{4\pi i Q\left(\bm{\rho_{\ell_1,\ell_2}}\right)+2\pi i B\left(A^{-1}A^*,\bm{\rho_{\ell_1,\ell_2}}\right)}\vartheta_{\bm{\rho_{\ell_1,\ell_2}},-\bm{a}-2\bm{\rho_{\ell_1,\ell_2}}-A^{-1}A^*}(\t).
\end{align*}
Using this, \eqref{l1l2}, and Theorem \ref{Zwmainthm1} (2), we have (as the relevant shift is $-2\bm{\rho_{\ell_1,\ell_2}}-A^{-1}A^{*}$ which lies in $A^{-1}\mathbb{Z}^2$)
\begin{equation}\label{ref4}
\vartheta_{\bm{\rho_{\ell_1,\ell_2}},-\bm{a}}(\t-2)=e^{-4\pi i Q\left(\bm{\rho_{\ell_1,\ell_2}}\right)}\vartheta_{\bm{\rho_{\ell_1,\ell_2}},-\bm{a}}(\t).
\end{equation}
Now, by \eqref{l1l2}, we have 
\begin{align*}
Q\left(\bm{\rho_{\ell_1,\ell_2}}\right)
=\begin{cases}
-\frac{1}{24}\ \ &\text{if}\ \ \ell_1=0, \ell_2\in\{0,1\},\\
\frac{11}{24}\ \ &\text{if}\ \ \ell_1=2, \ell_2\in\{0,1\}.
\end{cases}
\end{align*}
Plugging this into \eqref{ref4} gives 
\begin{equation*}
\vartheta_{\bm{\rho_{\ell_1,\ell_2}},-\bm{a}}(\t-2)=\z_{12}\vartheta_{\bm{\rho_{\ell_1,\ell_2}},-\bm{a}}(\t).
\end{equation*}
 Applying \eqref{ref3} with $\t$ replaced by $-\frac{1}{\t}$ and then using Proposition \ref{prop1} gives the second transformation.\qedhere 
\end{proof}

\section{Proof of Theorem \ref{BBthm}}\label{sec5}

We are now ready to finish the proof of Theorem \ref{BBthm}.

\begin{proof}[Proof of Theorem \ref{BBthm}]
First, by Proposition \ref{lemma2} and Proposition \ref{hathtrans}, $\widehat{M}$ satisfies
\[
\widehat{M}\left(\t+1\right)=\z_3\widehat{M}(\t),\qquad \widehat{M}\left(\frac{\t}{2\t+1}\right)=\z_{12}(2\t+1)\widehat{M}(\t).
\]
Note that $T$ and $(\begin{smallmatrix}
1&0\\
2&1
\end{smallmatrix})$ generate $\G_0(2)$ (see \cite[Proposition 2]{Chuman}). Moreover, by Corollary \ref{cor1}, $\widehat{M}(\t)$ is a $q$-series. 

We next show that $\widehat{M}(\t)$ has at most polynomial growth at the cusps $i\i$ and $0$ (these are the $2$ inequivalent cusps of $\G_0(2)$). This follows once we show this for $\widehat{H}(\t)$ and $\widehat{A}(\t)$. For $\widehat{H}(\t)$, the required growth follows from its realization as an indefinite theta function in Proposition \ref{prop1}. 

Next, we look at $\widehat{A}(\t)$. As $\t\to i\i$, using \eqref{mathcalE} and \eqref{defF}, we have
\[
\mathcal{E}(\t)=O(1),\quad T_j(\t)=O(1),\qquad \widehat{F}_j(\t)=O(1).
\]
Hence each term in the expression for $\widehat{A}(\t)$ in Proposition \ref{prop2} has at most polynomial growth at $i\infty$. This directly gives the claim for the cusp $i\i$. 

Using \eqref{mathcalE}, \eqref{eta}, Theorem \ref{Zthm}, and \eqref{Ttrans}, one obtains, as $\t\to 0$, 
\[
\widehat{A}\!\left(\t\right)=O(1)\]
and this concludes the claim for the cusp $0$. So $\widehat{M}^{12}$ is a holomorphic modular form of weight $12$ on $\G_0(2)$. By Theorem \ref{Sturm} (with $k=12$ and $N=2$) and Proposition \ref{Onoprop}, we need to show that $4$
Fourier coefficients of $\widehat{M}^{12}(\t)$ vanish. A direct computation shows that these coefficients vanish. Thus $\widehat{M}(\t)=0$. By Corollary \ref{cor1}, we conclude the proof.\qedhere 
\end{proof}


\section{Open questions}\label{sec6}

We conclude by stating several directions for future research:
\begin{enumerate}
	\item All of our functions are components of vector-valued harmonic Maass forms. Do the holomorphic parts of the other components have a combinatorial interpretation?
	\item In Conjecture \ref{ABConj3} even residue classes are studied. It might be worthwhile to study odd residue classes.
	\item Can some of the results be extended to other mock theta function to obtain objects with combinatorial meaning. 
	\item Is there a direct proof of Conjecture \ref{ABConj3}, e.g. by using Bailey pairs?
\end{enumerate}


\end{document}